\documentclass[12pt]{amsart}
\usepackage{amsfonts}
\usepackage{amsthm}
\usepackage{graphics}
\usepackage{indentfirst}
\usepackage{cite}
\usepackage{latexsym}
\usepackage{amsmath}
\usepackage{amssymb}
\usepackage[dvips]{epsfig}
\usepackage{amscd}
\usepackage{fancyhdr}
\newtheorem{theorem}{Theorem}[section]
\newtheorem{remark}{Remark}[section]

\newtheorem{lemma}[theorem]{Lemma}
\newtheorem{pro}{Proposition}[section]

\newcommand{\bt}{\begin{theorem}}
\newcommand{\bl}{\begin{lemma}}
\newcommand{\el}{\end{lemma}}
\newcommand{\et}{\end{theorem}}
\newcommand{\br}{\begin{remark}}
\newcommand{\er}{\end{remark}}

\newcommand{\la}{\label}

\newcommand{\bn}{\begin{eqnarray}}
\newcommand{\en}{\end{eqnarray}}
\newcommand{\bnn}{\begin{eqnarray*}}
\newcommand{\enn}{\end{eqnarray*}}

\newcommand{\ba}{\begin{aligned}}
\newcommand{\ea}{\end{aligned}}
\newcommand{\be}{\begin{equation}}
\newcommand{\ee}{\end{equation}}

\def\p{\partial}

\def\norm[#1]#2{\|#2\|_{#1}}

\def\la{\label}

\def\R{\mathbb{R}}

\begin{document}

\title{On some new global existence result of 3D Magnetohydrodynamic equations}

\author{Cheng He}
\address{Division of Mathematics, Department of Mathematical \& Physical Sciences, National Natural Science Foundation of
China, 100085,  People's Republic of China}
\email{hecheng@nsfc.gov.cn}

\author{Xiangdi Huang}
\address{NCMIS, Academy of Mathematics and Systems Science, CAS, Beijing 100190, P. R. China \ \ $\&$\ \ Department of Pure and Applied Mathematics, Graduate School of Information Sciences and Technology, Osaka University, Osaka, Japan}
\email{xdhuang@amss.ac.cn}

\author{Yun Wang}
\address{Department of Mathematics, Soochow University, 1 Shizi Street, Suzhou 215006, P.R. China}
\email{ywang3@suda.edu.cn}
\date{}

\thanks{The research of X.D.Huang is partially supported by NSFC 11101392. The research of Y. Wang is partially supported by NSFC 11241004.}

\maketitle

\begin{abstract} 
This paper  is devoted to the incompressible Magenetohydrodynamic equations in $\R^3$. We prove that if the difference between the magnetic field and the velocity is small initially then it will remain forever, thus results in global strong solution without smallness restriction on the size of initial velocity or magnetic field.  In other words, magnetic field can indeed regularize the Navier-Stokes equations, due to cancelation.


AMS: 35Q35, 35B65, 76N10
\end{abstract}

\vspace{5mm} \section{Introduction} \setcounter{equation}{0}

\vspace{2mm} The present paper is devoted to the study of the incompressible magnetohydrodynamics (MHD) equations
in $\mathbb{R}^3$,
\begin{equation}\label{sys}
\left\{\begin{array}{l} \displaystyle\frac{\partial u}{\partial t}
-\frac{1}{Re}
\Delta u + (u\cdot\nabla)u - S(B\cdot \nabla )B + \nabla (P+\frac{S}{2}|B|^2)  = 0,\\
[3mm] \displaystyle\frac{\partial B}{\partial t}
-\frac{1}{Rm} \Delta B + (u\cdot\nabla )B - (B\cdot\nabla )u = 0,\\
[2mm] {\rm div}~u = 0,~~~~{\rm div }~B = 0.
\end{array}\right.
\end{equation}
with the following initial conditions
 \begin{equation}\label{sys-2}
\left\{\begin{array}{l}u(x,0)=u_0(x)\ ,\\[2mm]
B(x,0)=B_0(x).
\end{array}\right.\end{equation}
Here $u, P, B$ are
non-dimensional quantities corresponding to the velocity of the
fluid, its pressure and the magnetic field, respectively. The
non-dimensional number $Re>0$ is the Reynolds number, $Rm >0 $ is
the magnetic Reynolds number and $S = M^2/(ReRm)$ with $M$ being the
Hartman number. For simplicity of writing, we can assume that $S=1$,
otherwise, let $\widetilde{B}(x,t)=\sqrt{S}B(x,t)$. And let $P$
denote the total pressure $P + S|B|^2/2$.

The MHD system (1.1) was studied by G. Duvaut and J.-L. Lions \cite{DL}. They established local existence and uniqueness of a solution in the classical Sobolev spaces $H^s(\R^N), s \ge N$.
But whether this unique local solution can exist globally for large initial data is a challenging open problem in mathematical fluid mechanics. Later, Sermange and Temam \cite{ST} showed the regularity for
weak solutions in the case of three dimension under the assumption
that $(u, B)$ belongs to $L^\infty(0, T; H^1(\mathbb{R}^3))$. The other
kinds of regularity criteria are established in \cite{CKS, wu00, wu02, wu04} and references therein. On the other hand,
Duvant-Lions\cite{DL} proved also global existence of the  strong solution for small initial data. For some extension, refer to \cite{MYZ, WD}. It should be noted that all these global existence results of smooth solutions require that both the velocity filed  and magnetic filed be sufficiently small.
This is mainly because that the MHD
equations share the same nonlinear convection structure as that of
the incompressible Navier-Stokes equations. However,  the relation between velocity field and magnetic field in the existence theory is not clear.

Recently, some efforts have been made to characterize the different roles
played by the velocity field and magnetic field  in the regularity
of weak solutions. Partial developments are achieved and partial
results are obtained in this direction.  Namely, it was
shown in \cite{he2} that a weak solution $(u, B)$ is smooth
provided the velocity field satisfies any one of the following
assumptions: 1) $u \in L^p(0, T; L^q(\R^3))$ with $1/p + 3/2q \leq
1/2$  for $q > 3$; 2)~$u \in C([0, T]; L^3(\R^3))$; 3)~$\nabla u \in
L^\alpha(0, T; L^\beta(\R^3))$  with $1/\alpha + 3/2\beta \leq 1$
for $3 \leq \beta < \infty$; 4)~Let $\omega(x, t) = \mbox{curl}~u(x,
t)$. There exist some positive constants $K$, $M$ and $\rho$, such that
$$\big|\omega(x+y, t) - \omega(x, t)\big| \leq K\big|\omega(x+y,
t)\big||y|^{1\over 2}$$ holds for any $t \in [0, T]$, if both
$|y|\leq \rho$ and $|\omega(x+y, t)| \geq M$. This result was then
generalized, for more references, see \cite{hw, KK, LD, wang, zhou}.
These regularity criteria imply that the velocity field of the fluid
seems to play a dominant role in the theory of regularity of weak
solutions in some sense.

On the other hand,  there is some evidence indicating that the magnetic field should have some dissipation, due to the numerical simulations of Politano et in \cite{PPS} and
the observations of space and laboratory plasmas alike in \cite{H}. Then
the solutions to the incompressible magnetohydrodynamic equations
should exhibit a greater degree of regularity than does an ordinary
incompressible Navier-Stokes equations, in some sense. In this direction, Bardos-Sulem-Sulem\cite{BSS} first established the global strong classical solution for the invscid MHD equations with strong magenetic field. Inspired by the results in \cite{H, PPS}, we study the concelation between the velocity field and the magnetic field. 

First, reformulate the equation (\ref{sys}) using Elsasser's variables $W^+,W^-$ as follows,
$$W^+=u+B,\ \ \ W^-=u-B,\ \ \ W^+_0=u_0+B_0,\ \ W^-_0=u_0-B _0.$$
Then MHD equations (\ref{sys}-(\ref{sys-2}) can be re-written as:
\begin{equation}\label{sys-new}
\left\{\begin{array}{l} \displaystyle\frac{\partial W^+}{\partial t}
-\kappa
\Delta W^+ - \lambda\Delta W^- + (W^-\cdot\nabla)W^+ + \nabla P = 0,\\
[3mm] \displaystyle\frac{\partial W^-}{\partial t}
-\kappa \Delta W^- -\lambda\Delta W^+ + (W^+\cdot\nabla )W^- +\nabla P = 0,\\
[2mm] {\rm div}~W^+ = 0,~~~~{\rm div }~W^- = 0.
\end{array}\right.
\end{equation}
with the following initial conditions \begin{equation}\label{sys-new-1}
\left\{\begin{array}{l}W^+(x,0)=W^+_0(x)\ ,\\[2mm]
W^-(x,0)=W^-_0(x).\end{array}\right.\end{equation} Here $$\kappa  =
\frac{1}{2Re}+\frac{1}{2Rm}, \  \  \lambda =
\frac{1}{2Re}-\frac{1}{2Rm}.$$
Note that
\be
\kappa>|\lambda|.
\ee
In this paper, we will show that magnetic field can regularize the Navier-Stokes equations. More precisely, there exists a unique global strong solution to 3D MHD equations with large initial velocity as long as $\frac{|\lambda|}{\kappa}\ll 1$ and the magnetic field is comparable to the velocity initially.

%%%%%%%%%%%%%%%%%%% Main result %%%%%%%%%%%%%%%%%%%%%%%%%%%%%%%%%%%%%%%%%%%%%%%%%%%%%
 \vspace{5mm}

\section{Main Result}
\setcounter{equation}{0} \vspace{2mm}

Before stating our main results, we introduce some function spaces.
Let $C_{0,\sigma}^\infty(\mathbb{R}^3)$ denote the set of all
$C^\infty$ vector-valued functions $\phi$ with compact support in $\mathbb{R}^3$,
such that $\mbox{div}\phi=0$. $L^\beta(\R^3), W^{k, \beta}(\R^3)$$(1\leq \beta \leq \infty)$, $H^s(\R^3)(s >0 )$  are
the standard Sobolev spaces. The fractional-order homogeneous Sobolev
space $\dot{H}^s(\mathbb{R}^3)$ $(s>0)$ is defined as the space of tempered
distributions $u$ over $\mathbb{R}^3$  for which the Fourier transform
$\mathcal{F}u$ belongs to $L_{loc}^1(\mathbb{R}^3)$ and which satisfy
\be \nonumber
    \|u\|_{\dot{H}^s(\mathbb{R}^3)}^2 := \int_{\mathbb{R}^3}
    |\xi|^{2s} |\mathcal{F} u(\xi)|^2 d\xi < \infty ~.
\ee
Let $L^\beta_\sigma(\mathbb{R}^3)$ and 
$\dot{H}^{\frac12}_\sigma(\R^3)$ be the closure of
$C_{0,\sigma}^\infty(\mathbb{R}^3)$ with the respect to the
 $L^\beta$-norm and $\dot{H}^{\frac12}$-norm.

This paper is devoted to the existence of global strong solutions to the problem (\ref{sys})-(\ref{sys-2}). We provide a new smallness condition on the initial data $W^-_0(x)$ and  $\frac{|\lambda|}{\kappa}$ rather than the initial velocity and magnetic field $(u_0,B_0)$.

To illustrate our main idea in a clear way, we study a special case first, for which $Re=Rm$, and prove the following theorem. 
\begin{theorem}\la{t1}
Let $(W^+_0,W^-_0) \in L_{\sigma}^3(\R^3)$ and $\lambda=0$.  Then there exist two generic constants $\epsilon_0$
and $C_0$, independent of initial data and $\kappa$, such that if
\be\la{condition-01}
\kappa^{-3} \|W^-_0\|_{L^3}^3 \cdot \exp \left\{ C_0 \kappa^{-3} \|W^+_0\|_{L^3}^3\right\}< \epsilon_0\, ,
\ee
or
\be\la{condition-02}
\kappa^{-3} \|W^+_0\|_{L^3}^3 \cdot \exp \left\{ C_0 \kappa^{-3} \|W^-_0\|_{L^3}^3\right\}< \epsilon_0\, ,
\ee
then the system (\ref{sys-new})-(\ref{sys-new-1}) admits a global strong solution $$(W^+,W^-)\in C([0, \infty), L_{\sigma}^3(\R^3))\cap
C((0,\infty), W^{2,3}(\R^3)).$$
\end{theorem}

\vspace{2mm}\begin{remark}
Condition (\ref{condition-01}) holds for
\be\la{w-1}
W^-_0(x) = 0\, ,
\ee
while condition (\ref{condition-02}) holds for
\be\la{w-2}
W^+_0(x) = 0\, .
\ee

Let's say a few words on this special case. Taking (\ref{w-1}) for example, one can easily deduce that
\be
W^+(x,t) =\nabla P=0\, ,
\ee
and $W^-(x,t)$ is a solution to the following heat equations
\be\la{stokes}
\p_tW^- -\kappa\triangle W^-=0\, .
\ee
 There is no doubt that \eqref{stokes} has a global solution, which is smooth when $t > 0$.
\end{remark}

\begin{remark}
Very recently, a similar version of Theorem \ref{t1} was indicated in Lei-Lin-Zhou's paper\cite{LLZ}. 
\end{remark}

Next, we generalize the same idea to a general case, where $Re \neq Rm$.
\begin{theorem}\la{t2}
Let $(W^+_0,W^-_0) \in \dot{H}^{\frac12}_{\sigma}(\R^3)$.  Then there exist two generic positive constants $\epsilon_0<\frac12$ and $C_0$,  independent of initial data, $\kappa$ and $\lambda$, such that if
\be\la{condition2-01}
 \left( \kappa^{-2} \|W^-_0\|_{\dot{H}^{\frac12}}^2 +  \frac{\lambda^2}{\kappa^2} \left( \kappa^{-2}  \|W^+_0\|_{\dot{H}^{\frac12}}^2 + \frac{\lambda^2}{\kappa^2}\right)\right) \exp \left\{C_0 \left(\kappa^{-4} \|W^+_0\|_{\dot{H}^{\frac12}}^4 + \frac{\lambda^4}{\kappa^4}\right)\right\}< \epsilon_0 \,
\ee
or
\be\la{condition2-02}
\left( \kappa^{-2} \|W^+_0\|_{\dot{H}^{\frac12}}^2 +  \frac{\lambda^2}{\kappa^2} \left( \kappa^{-2}  \|W^-_0\|_{\dot{H}^{\frac12}}^2 + \frac{\lambda^2}{\kappa^2}\right)\right) \exp \left\{C_0 \left(\kappa^{-4} \|W^-_0\|_{\dot{H}^{\frac12}}^4 + \frac{\lambda^4}{\kappa^4}\right)\right\}< \epsilon_0 \, ,
\ee
then the system (\ref{sys-new})-(\ref{sys-new-1}) admits a global strong solution $$(W^+,W^-)\in C([0, \infty), \dot{H}^{\frac12}_{\sigma}(\R^3))\cap
C((0,\infty), \dot{H}^2 (\R^3)).$$
\end{theorem}

\vspace{2mm}
\begin{remark}
Either condition (\ref{condition2-01}) or (\ref{condition2-02}) automatically implies
\be
\frac{|\lambda|}{\kappa}\le \epsilon_0^{\frac{1}{4}}\ll 1,
\ee
which corresponds to $\frac{|Re-Rm|}{Re+Rm}\ll 1$ in original MHD system.  Indeed, in astrophysical magnetic phenomena, both the Reynolds number and the magnetic Reynolds number are huge and the difference between $Re$ and $Rm$ is not critical. In this aspect our assumption is reasonable.
\end{remark}

\begin{remark}
From mathematical point of view, it's important to study the regularizing effect of magnetic field to the Navier-Stokes equations. As shown in Theorem 2.1-2.2, one type of regularizing effects is due to cancelation. Consequently,  synchronous diffusion speed is a good candidate for maintaining the effect of cancelation all the time which turns out to ask $\frac{|Re-Rm|}{Re+Rm}\ll 1$.
\end{remark}

\begin{remark}
There is no smallness condition imposed on the initial velocity, it indicates that one can generate solutions that are smooth for all the time $t>0$ for the large initial velocity as long as magnetic field and velocity are comparable initially and $\frac{|\lambda|}{\kappa}\ll 1$. In other words, magnetic field can indeed regularize the Navier-Stoke equations.
\end{remark}

\begin{remark}
In fact, Theorem 2.2 can not cover Theorem 2.1 completely. We require the initial data belongs to $\dot{H}^{\frac12}(\R^3)$, instead of $L^3(\R^3)$. Hence it is left open whether a real generalized version of Theorem 2.1 can be derived. 
\end{remark}

\begin{remark}
  Due to the symmetric structure of the system (\ref{sys-new}), we need only to prove Theorems \ref{t1} and \ref{t2} under condition (\ref{condition-01})and (\ref{condition2-01}) respectively. 
\end{remark}

%%%%%%%%%%%%%%%%Proof of Thm 1%%%%%%%%%%%%%%%%%%%%%%%%%%%%%%%%%%%%%%%%%%
\section{Proof of Theorem \ref{t1}}

Since $(W^+_0, W^-_0) \in L^3_\sigma(\R^3)$ is equivalent to $(u_0, B_0) \in L^3_\sigma (\R^3)$, it is well-known, that
there are $T_0 > 0$ and a unique strong solution $(u, B)$ to the
MHD equations in $(0, T_0]$ and the solution is classical when $t>0$. Hence, in
the following, we assume that the solution $(W^+, W^-)$ is sufficiently
smooth on $[0, T]$ and deduce the uniform a priori strong estimates
under the assumptions of Theorem \ref{t1}, which guarantee the extension of the local strong solution.

Here and thereafter, $C, C_1$ will denote a generic constant which is independent of $\lambda$, $\kappa$, the  initial data $(W^+_0,W^-_0)$ and time $T$.

\subsection{A priori Estimates}
Given a strong solution $(W^+ , W^-)$ on $\mathbb{R}^3 \times [0, T]$ for $T<T_0$, define
\be\la{at}
A^-(T) = \kappa^{-3}\sup_{0\leq t\leq T} \|W^-(t)\|_{L^3}^3.
\ee
Actually, we have the following proposition.

\begin{pro}\label{pro}
Let $(W^+, W^-)$ be a strong solution to
\eqref{sys-new}-\eqref{sys-new-1} in $[0,T]\times\R^3$, then there
exist two positive constants $\epsilon_0$ and $C_0$,  such that if
\be\label{assumption}
   A^-(T)  \leq 2\epsilon_0\, ,
\ee
then it in fact holds that
\[
  A^-(T)  \leq \epsilon_0 \ \ \ \mbox{and}\ \ \ \|W^+\|_{L^3(0, T; L^9)} \leq C \kappa^{-\frac13} \|W^+_0\|_{L^3}\, ,
\]
provided
\be\la{condition-1}
\kappa^{-3} \|W^-_0\|_{L^3}^3 \cdot\exp{\{ C_0 \kappa^{-3} \|W^+_0\|_{L^3}^3\}}  \leq \epsilon_0\, .
\ee
\end{pro}
The remainder of this subsection consists in proving this key result.

%Multiplying the first and second equation in (\ref{sys-new}) by $W^+$ and $W^-$ respectively, integrating over %$\R^3$, one has the following energy estimates:
%\begin{lemma}\la{le-1}(Energy estimates)
%\be\la{new-1}
%\left\{
%\ba
%& \frac12 \|W^+\|_{L^2}^2 + \kappa\int_0^T\|\nabla W^+\|_{L^2}^2\, dt \le \frac12\|W^+_0\|_{L^2}^2\\
%& \frac12 \|W^-\|_{L^2}^2 + \kappa\int_0^T\|\nabla W^-\|_{L^2}^2\, dt \le \frac12 \|W^-_0\|_{L^2}^2
%\ea
%\right.
%\ee
%\end{lemma}

\begin{lemma}\la{le-2}
Let $(W^+, W^-)$ be a strong solution to \eqref{sys-new}-\eqref{sys-new-1} in $[0,T]\times\R^3$. There exist two positive constant $\epsilon_0$ and $C_0$, such that if
\be \la{assumption-1}
A^-(T)\le 2\epsilon_0,
\ee
then it holds that
\be\la{cru-1}
\sup_{0\leq t\leq T} \|W^+(t)\|_{L^3}^3 + \kappa \int_0^T \left( \left\| | W^+|^{\frac12} \nabla W^+ \right\|_{L^2}^2 + \left\| \nabla \left(|W^+|^{\frac32} \right)\right\|_{L^2}^2 \right)\, dt
\leq \|W_0^+\|_{L^3}^3,
\ee
and
\be\la{cru-2}
\sup_{0\leq t \leq T}\|W^-(t)\|_{L^3}^3 \leq \|W^-_0\|_{L^3}^3 e^{C_0 \kappa^{-3} \|W^+_0\|_{L^3}^3}.
\ee
\end{lemma}

\begin{proof} {\it Step 1.}\  Multiplying the first equation of (\ref{sys-new}) by
$ 3 |W^+| W^+$ and integrating over $\R^3$, we get
\be\label{ww-0} \ba
&\frac{d}{dt} \|W^+\|_{L^3}^3 + 3 \kappa \int |W^+| |\nabla W^+|^2\, dx + \frac43 \kappa \int \left| \nabla \left(|W^+|^{\frac32}\right) \right|^2\, dx \\
& \leq  3 \int \left| P\cdot {\rm div} \left( |W^+| W^+\right)  \right| \, dx\\
& \leq  C \|P\|_{L^{\frac94}} \|W^+\|_{L^9}^{\frac12} \left\| |W^+|^{\frac12} \nabla W^+ \right\|_{L^2}.
\ea \ee
Note that the equation for $W^+$ can be written as
\be \la{ww-3}
(\partial_t W^+ -\kappa \Delta W^+) + \nabla P=  - {\rm div} (W^- \otimes W^+),
\ee
where the left-hand side is viewed as the Helmholtz-Weyl decomposition of the right one.
From this equation, one has
\be\la{ww-4}
\ba
\|P\|_{L^{\frac94}} & \leq C \|W^- \otimes W^+\|_{L^{\frac94}}\\
& \leq  C \|W^-\|_{L^3} \|W^+\|_{L^9}.
\ea
\ee

Insert the estimate \eqref{ww-4} into \eqref{ww-0}, then
\be \la{ww-5} \ba
&\frac{d}{dt} \|W^+\|_{L^3}^3 + 3 \kappa \int |W^+| |\nabla W^+|^2\, dx + \frac43 \kappa \int \left| \nabla \left(|W^+|^{\frac32}\right) \right|^2\, dx \\
&\leq C \|W^-\|_{L^3} \|W^+\|_{L^9}^{\frac32} \left\| |W^+|^{\frac12} \nabla W^+ \right\|_{L^2}\\
& \leq  C \|W^-\|_{L^3} \left\|  |W^+|^{\frac32} \right\|_{L^6} \left\| |W^+|^{\frac12} \nabla W^+ \right\|_{L^2} \\
& \leq C \|W^-\|_{L^3} \left\|\nabla \left(|W^+|^{\frac32}\right)\right\|_{L^2} \left\| |W^+|^{\frac12} \nabla W^+ \right\|_{L^2}\\
& \leq  C_1 \|W^-\|_{L^3} \left\|\nabla \left(|W^+|^{\frac32}\right)\right\|_{L^2}^2 + C_1  \|W^-\|_{L^3} \left\| |W^+|^{\frac12} \nabla W^+ \right\|_{L^2}^2
\ea
\ee
where the Sobolev embedding inequality and Young's inequality were used.

 Now if we choose $\epsilon_0$ small enough such that
\be \la{ww-6}
A^-(T)^{\frac13} \leq \left(2\epsilon_0\right)^{\frac13} \leq \kappa/(3C_1),
\ee
then \eqref{cru-1} follows immediately.

{\it Step 2.} Multiplying the second equation of (\ref{sys-new}) by $3|W^-|W^-$,
with the help of integration by parts, we have
\be\la{ww-8}
\ba
&\frac{d}{dt}\|W^-\|_{L^3}^3 +  3 \kappa \int |W^-| |\nabla W^-|^2\, dx + \frac43 \kappa \int \left| \nabla \left(|W^-|^{\frac32}\right) \right|^2\, dx \\
& \leq  C \|P \|_{L^{\frac94}} \| W^-\|_{L^9}^{\frac12} \left\| |W^-|^{\frac12} \nabla W^- \right\|_{L^2}.
\ea
\ee
Substituting the estimate \eqref{ww-4} into \eqref{ww-8} and employing the Sobolev embedding inequality, we have
\be \la{ww-9} \ba
&\frac{d}{dt}\|W^-\|_{L^3}^3 +  3 \kappa \int |W^-| |\nabla W^-|^2\, dx + \frac43 \kappa \int \left| \nabla \left(|W^-|^{\frac32}\right) \right|^2\, dx \\
& \leq C \|W^-\|_{L^3} \| W^+ \|_{L^9} \| W^-\|_{L^9}^{\frac12} \left\| |W^-|^{\frac12} \nabla W^- \right\|_{L^2}\\
&\leq C \|W^-\|_{L^3} \left\| \nabla \left(|W^+|^{\frac32}\right)\right\|_{L^2}^{\frac23}  \left\| \nabla \left(|W^-|^{\frac32}\right)\right\|_{L^2}^{\frac13} \left\| |W^-|^{\frac12} \nabla W^- \right\|_{L^2}.
\ea
\ee
One can easily deduce from \eqref{ww-9} after using Young's inequality that
\be\la{ww-10}
\frac{d}{dt}\|W^-\|_{L^3}^3  \leq C \kappa^{-2} \| W^-\|_{L^3}^3 \left\| \nabla \left(|W^+|^{\frac32}\right)\right\|_{L^2}^2.
\ee

Hence, in view of \eqref{cru-1}, one has
\be\la{ww-11}
\sup_{0\leq t\leq T} \|W^-(t)\|_{L^3}^3 \leq \|W^-_0\|_{L^3}^3 e^{C_0 \kappa^{-3} \|W_0^+\|_{L^3}^3} .
\ee
This finishes the proof of Lemma \ref{le-2}.
\end{proof}

It follows from Lemma \ref{le-2} that
\be A^-(T) \leq \kappa^{-3} \|W^-_0\|_{L^3}^3 e^{C_0 \kappa^{-3} \|W_0^+\|_{L^3}^3} \leq \epsilon_0,
\ee
if  \eqref{condition-1} holds.  And also the estimate
\be
\|W^+\|_{L^3(0, T; L^9)} \leq  C \kappa^{-\frac13} \|W^+_0\|_{L^3}.
\ee
is implied in Lemma \ref{le-2}, more precisely, \eqref{cru-1}.
Hence we finishes the proof of Proposition \ref{pro}.

\subsection{Proof of Theorem \ref{t1} }
With the a priori estimates in previous subsection in hand, we are prepared
for the proof of Thorem~\ref{t1}.
\begin{proof}
In view of classical results, there exists a $T_*>0$ such
that the Cauchy problem \eqref{sys-new}-\eqref{sys-new-1} has a unique local
strong solution $( W^+, W^-)$ on $\mathbb{R}^3 \times (0, T_*]$. We will show that this local solution can be extended to a global one provided condition (\ref{condition-01}) holds.

Since the local strong solution is continous in $L^3$,  there exists a $T_1 \in (0, T_*)$ such that \eqref{assumption} holds for $T= T_1$. So we set
\[
   \bar{T}= \sup \{ T |\ (W^+, W^-) \ \mbox{is a strong solution on}\
     \mathbb{R}^3 \times (0, T] \ \mbox{and}\  A^-(T)\leq 2\epsilon_0 \} ~,
\]
and 
\[
T^* = \sup\{ T| \ (W^+, W^-)\ \mbox{ is a strong solution on } \ \R^3 \times (0, T]\}~.
\]
Obviously, $\bar{T} \leq T^*.$ However, in fact it follows from Proposition \ref{pro} that
\be
\bar{T} = T^*,
\ee
if condition \eqref{condition-01} holds.
 We claim that $T^* = \infty$, for which we will argue  by contradiction. Suppose $T^* < \infty$, as proved in Proposition \ref{pro},
\be
\sup_{0\leq T < T^*}\|W^+ \|_{L^3(0, T; L^9)} \leq C \kappa^{-\frac13} \|W_0^+\|_{L^3}^3,
\ee
which guarantees that the local solution will not blow up at $T^*$, according to the blowup criterion in \cite{hw2} . Hence it contradicts to the definition of $T^*$.

\end{proof}

%%%%%%%%%%%%%%%%%%%proof of Thm 2  %%%%%%%%%%%%%%%%%%%%%%%%%%%%%%%%%%%%%%%%%%%%%%%%%%%%%
\section{Proof of Theorem \ref{t2} }
The main idea of the proof for Theorem \ref{t2} is the same as that of Theorem \ref{t1}. The computations and techniques here are a bit more complicated. Since $(W^+_0, W^-_0) \in \dot{H}^{\frac{1}{2}}_\sigma(\R^3)$ is equivalent to $(u_0, B_0) \in \dot{H}^{\frac{1}{2}}_\sigma(\R^3)$, 
there are $T_0 > 0$ and a unique strong solution $(u, B)$ to the
MHD equations in $(0, T_0]$ and the solution is classical when $t>0$. Hence, in
the following, we assume that the solution $(W^+, W^-)$ is sufficiently
smooth on $[0, T]$ and deduce the uniform a priori strong estimates
under the assumptions of Theorem \ref{t2}, which guarantee the extension of the local strong solution.

For simplicity of writing, we first scale $W^+, W^-, P$ to $V^+, V^-, \tilde{P}$ as follows,
\begin{equation}\la{change}
\left\{\begin{array}{l}
V^+(x, t) = \kappa^{-1}W^+(x, \kappa^{-1} t)\ ,\\[2mm]
V^-(x, t) = \kappa^{-1} W^-(x, \kappa^{-1} t)\ ,\\[2mm]
\tilde{P}(x, t)= \kappa^{-2}P (x, \kappa^{-1} t).
\end{array}\right.\end{equation}

Then the system for $(V^+, V^-, \tilde{P})$ becomes
\be\la{sys-new-new} \left\{
\ba &  \displaystyle\frac{\partial V^+}{\partial t}
- \Delta V^+ - \frac{\lambda}{\kappa} \Delta V^- + (V^-\cdot\nabla)V^+ + \nabla \tilde{P}  = 0,\\
& \displaystyle\frac{\partial V^-}{\partial t}
- \Delta V^- -\frac{\lambda}{\kappa}\Delta V^+ + (V^+\cdot\nabla )V^- +\nabla \tilde{P} = 0,\\
& {\rm div}~V^+ = {\rm div }~V^- = 0.
\ea \right.
\ee
with the following initial conditions
\begin{equation}\label{sys-new-new-1}
\left\{\begin{array}{l}V^+(x,0)=\kappa^{-1} W^+_0(x)\ ,\\[2mm]
V^-(x,0)=\kappa^{-1}W^-_0(x).\end{array}\right.\end{equation}

Given a strong solution $(V^+ , V^-, \tilde{P})$ in $\mathbb{R}^3 \times [0, T]$ for $T<T_0$, define
\be\la{at2}
A^-(T) = \sup_{0\leq t\leq T} \|V^-(t, \cdot)\|_{\dot{H}^{\frac12}}^2 + \int_0^T \|V^-(t, \cdot)\|_{\dot{H}^{\frac32}}^2 \, dt .
\ee
Similarly, we have the following proposition.
\begin{pro}\label{pro2}
Let $(V^+, V^-, \tilde{P})$ be a strong solution to \eqref{sys-new-new}-\eqref{sys-new-new-1} in $\R^3 \times [0, T]$, then there exist two positive constants $\epsilon_0< \frac12 $ and $C_0$, such that if
\be\label{assumption2}
   A^-(T)  \leq 2\epsilon_0\, ,
\ee
then it in fact holds that
\[
  A^-(T)  \leq \epsilon_0\, ,
\]
provided
\be\la{condition2-1}
\ba
&\left( \kappa^{-2} \|W^-_0\|_{\dot{H}^{\frac12}}^2 +  \frac{\lambda^2}{\kappa^2} \left( \kappa^{-2}  \|W^+_0\|_{\dot{H}^{\frac12}}^2 + \frac{\lambda^2}{\kappa^2}\right)\right)\\ & \cdot \exp \left\{C_0 \left(\kappa^{-4} \|W^+_0\|_{\dot{H}^{\frac12}}^4 + \frac{\lambda^4}{\kappa^4}\right)\right\}< \epsilon_0 \, .
\ea
\ee
\end{pro}

%%%%%%%%%%%%%%%%%%%%%%%Lemma %%%%%%%%%%%%%%%%%%%%%%%%%%%%%
\begin{lemma}\la{le2-2}
Let $(V^+, V^-)$ be a strong solution to \eqref{sys-new-new}-\eqref{sys-new-new-1} in $\R^3\times [0, T]$.
There exist two positive constants $\epsilon_0<\frac12$ and $C_0$, such that if
\be \la{assumption-1}
A^-(T)\le 2\epsilon_0,
\ee
then it holds that
\be\la{cru2-1}
\sup_{0\leq t\leq T} \|V^+(t)\|_{\dot{H}^{\frac12}}^2 +  \int_0^T \|V^+\|_{\dot{H}^{\frac32}}^2 \, dt
\leq \kappa^{-2} \|W^+_0\|_{\dot{H}^{\frac12}}^2 + \frac{\lambda^2}{\kappa^2}\, .
\ee
and
\be\la{cru2-2} \ba
&\ \ \ \ \ \ \sup_{0\leq t\leq T} \|V^-(t)\|_{\dot{H}^{\frac12}}^2 + \int_0^T \|V^-\|_{\dot{H}^{\frac32}}^2 \, dt\\
&\leq
\left( \kappa^{-2} \|W^-_0\|_{\dot{H}^{\frac12}}^2 +  \frac{\lambda^2}{\kappa^2} \left( \kappa^{-2}  \|W^+_0\|_{\dot{H}^{\frac12}}^2 + \frac{\lambda^2}{\kappa^2}\right)\right) \exp \left\{C_0 \left(\kappa^{-4} \|W^+_0\|_{\dot{H}^{\frac12}}^4 + \frac{\lambda^4}{\kappa^4}\right)\right\}.
\ea \ee
\end{lemma}

\begin{proof} {\it Step 1.}\  Multiplying the first equation of (\ref{sys-new-new}) by
$ (-\Delta)^{\frac12} V^+$ and integrating over $\R^3$, we get
\be\label{ww2-0}
\ba
&\frac12 \frac{d}{dt} \|V^+\|_{\dot{H}^{\frac12}}^2  +  \|V^+\|_{\dot{H}^{\frac32}}^2  \\
&=  \frac{\lambda}{\kappa} \int \Delta V^- \cdot (-\Delta)^{\frac12} V^+ \, dx - \int (V^- \cdot \nabla) V^+ \cdot (-\Delta)^{\frac12} V^+ \, dx\\
& \leq \frac{|\lambda|}{\kappa} \|V^-\|_{\dot{H}^{\frac32}} \|V^+\|_{\dot{H}^{\frac32}}  +
\|V^-\|_{L^3} \|\nabla V^+\|_{L^3}^2 \\
& \leq   \frac14 \|V^+\|_{\dot{H}^{\frac32}}^2 +   \frac{\lambda^2}{\kappa^2}\|V^-\|_{\dot{H}^{\frac32}}^2+
C_1 \|V^-\|_{\dot{H}^{\frac12}} \|V^+\|_{\dot{H}^{\frac32}}^2\, ,
\ea
\ee
where we used the Sobolev embedding inequality.

If we choose $\epsilon_0 $ small enough  such that
\be
\epsilon_0^{\frac{1}{2}} \leq \min\left\{ \frac{1}{8C_1}, \frac12 \right\},
\ee
then it holds that
\be \la{ww2-1}
\sup_{0\leq t\leq T} \|V^+\|_{\dot{H}^{\frac12}}^2 + \int_0^T \|V^+\|_{\dot{H}^{\frac32}}^2\, dt
\leq \|V^+_0\|_{\dot{H}^{\frac12}}^2 + \frac{\lambda^2}{\kappa^2}
= \kappa^{-2} \|W^+_0\|_{\dot{H}^{\frac12}}^2 + \frac{\lambda^2}{\kappa^2}\, .
\ee

{\it Step 2.}\ Multiplying the second equation of (\ref{sys-new-new}) by
$ (-\Delta)^{\frac12} V^-$ and integrating over $\R^3$, we get
\be \la{ww2-2} \ba
&\frac12 \frac{d}{dt} \|V^-\|_{\dot{H}^{\frac12}}^2  +  \|V^-\|_{\dot{H}^{\frac32}}^2  \\
&=  \frac{\lambda}{\kappa} \int \Delta V^+ \cdot (-\Delta)^{\frac12} V^- \, dx - \int (V^+ \cdot \nabla) V^- \cdot (-\Delta)^{\frac12} V^- \, dx\\
& \leq \frac{|\lambda|}{\kappa} \|V^-\|_{\dot{H}^{\frac32}} \|V^+\|_{\dot{H}^{\frac32}}  +
\|V^+\|_{L^6} \|\nabla V^-\|_{L^3} \|\nabla V^-\|_{L^2} \\
& \leq   \frac14 \|V^- \|_{\dot{H}^{\frac32}}^2 +   \frac{\lambda^2}{\kappa^2}\|V^+\|_{\dot{H}^{\frac32}}^2+
C \| \nabla V^+\|_{L^2} \|V^-\|_{\dot{H}^{\frac12}}^{\frac12} \|V^-\|_{\dot{H}^{\frac32}}^{\frac32}\\
& \leq \frac12 \|V^- \|_{\dot{H}^{\frac32}}^2 +   \frac{\lambda^2}{\kappa^2}\|V^+\|_{\dot{H}^{\frac32}}^2
+ C \|\nabla V^+\|_{L^2}^4 \|V^-\|_{\dot{H}^{\frac12}}^2.
\ea
\ee
where the interpolation inequality and Sobolev embedding inequality were employed.

It follows from the Gronwall's inequality that
\be \la{ww2-3} \ba
&\sup_{0\leq t\leq T} \|V^-(t)\|_{\dot{H}^{\frac12}}^2 + \int_0^T \|V^-\|_{\dot{H}^{\frac32}}^2 \, dt\\
&\leq e^{C \int_0^T \|\nabla V^+\|_{L^2}^4\, dt} \left( \kappa^{-2} \|W^-_0\|_{\dot{H}^{\frac12}}^2
+ \frac{2\lambda^2}{\kappa^2} \|V^+ \|_{L^2(0, T; \dot{H}^{\frac32})}^2   \right).
\ea \ee

Using the interpolation theory, one has
\be \int_0^T \|\nabla V^+\|_{L^2}^4 \, dt \leq C \|V^+\|_{L^\infty(0, T; \dot{H}^{\frac12})}^2  \|V^+\|_{L^2(0, T; \dot{H}^{\frac32})}^2\, .
\ee
Substitute this inequality into \eqref{ww2-3}, and one can easily deduce after \eqref{ww2-1} that
\be \la{ww2-4} \ba
&\sup_{0\leq t\leq T} \|V^-(t)\|_{\dot{H}^{\frac12}}^2 + \int_0^T \|V^-\|_{\dot{H}^{\frac32}}^2 \, dt\\
& \leq
\left( \kappa^{-2} \|W^-_0\|_{\dot{H}^{\frac12}}^2 +  \frac{\lambda^2}{\kappa^2} \left( \kappa^{-2}  \|W^+_0\|_{\dot{H}^{\frac12}}^2 + \frac{\lambda^2}{\kappa^2} \right)\right) \exp \left\{C \left(\kappa^{-4} \|W^+_0\|_{\dot{H}^{\frac12}}^4 + \frac{\lambda^4}{\kappa^4}\right)\right\},
\ea
\ee
which is our desired estimate \eqref{cru2-2}.

\end{proof}

It follows from Lemma \ref{le2-2} that $A^-(T) \leq \epsilon_0$, if  \eqref{condition2-1} holds. Hence we finish the proof of Proposition \ref{pro2}. The remaining part of the proof for Theorem \ref{t2} follows the same lines as in Section 3, so we omit the details.

%%%%%%%%%%%%%%%%%%%%%%%%%%%%%%proof of thm 2.3%%%%%%%%%%%%%%%%%%%%%%%%%
%%\subsection{Proof of Theorem \ref{t1} }
%%With the a priori estimates in previous subsection in hand, we are prepared
%%for the proof of Theorem~\ref{t1}.
%%\begin{proof}
%%In view of classical results, there exists a $T_*>0$ such
%%that the Cauchy problem \eqref{sys-new-new}-\eqref{sys-new-new-1} has a unique local
%%strong solution $( V^+, V^-)$ on $\mathbb{R}^3 \times (0, T_*]$. We will show that this local solution can be %%extended to a global one provided condition \eqref{condition-01} holds.

%%Since the local strong solution is continous in $\dot{H}^{\frac12}$ and due to the time-integrability of the %%$\dot{H}^{\frac32}$-norm of the solution,  there exists a $T_1 \in (0, T_*)$ such that \eqref{assumption} holds for %%$T= T_1$. Set
%%\[
%%   \bar{T}= \sup \{ T |\ (V^+, V^-) \ \mbox{is a strong solution in}\
%%     \mathbb{R}^3 \times (0, T] \ \mbox{and}\  A^-(T)\leq 2\epsilon_0 \} ~.
%%\]
%%It follows from Proposition \ref{pro} that
%%\be
%%\bar{T} = T^*,
%%\ee
%%if condition \eqref{condition-01} holds.
%% We claim that $T^* = \infty$, for otherwise $T^* < \infty$, which
%% we will argue  by contradiction. As proved in Proposition \ref{pro},
%%\be
%%\sup_{0\leq T < T^*}\|\nabla V^+ \|_{L^4(0, T; L^2)} \leq C \left(\kappa^{-1}\|W_0^+\|_{\dot{H}^{\frac12}} +
%%\frac{|\lambda|}{\kappa}\right) ,
%%\ee
%%which guarantees that the local solution will not blow up at $T^*$, according to the blowup criterion in \cite{hw2} . %%Hence it contradicts to the definition of $T^*$.

%%\end{proof}

\end{document}